\documentclass[11pt]{amsart}

\usepackage{amssymb,amsthm,amsfonts,amsmath}
\usepackage{a4wide}
\usepackage{amsrefs}
\usepackage[utf8]{inputenc}
\usepackage[T1]{fontenc}
\usepackage{xcolor}
\usepackage[colorlinks]{hyperref}
 \usepackage{ulem}
 \usepackage{graphicx}

\newcommand{\R}{{\mathbb R}}

\renewcommand{\phi}{{\varphi}}

\newcommand{\Per}{\mathrm{Per}}

\newtheorem{theorem}{Theorem}[section]

\newtheorem{lemma}[theorem]{Lemma}

\theoremstyle{definition}

\newcommand{\eps}{\varepsilon}

\newcommand{\EEE}{\color{black}}

\begin{document}
 
\author{Annalisa Cesaroni}
\address{Dipartimento di Matematica "Tullio Levi-Civita", Universit\`{a} di Padova, Via Trieste 63, 35131 Padova, Italy}
\email{annalisa.cesaroni@unipd.it}
\author{Ilaria Fragal\`a}
\address{Dipartimento di Matematica, Politecnico di Milano, Piazza Leonardo da Vinci 32, 20133 Milano, Italy}
\email{ilaria.fragala@polimi.it}
\author{Matteo Novaga}
\address{Dipartimento di Matematica, Universit\`{a} di Pisa, Largo Bruno Pontecorvo 5, 56127 Pisa, Italy}
\email{matteo.novaga@unipi.it}

\begin{abstract} 
We show that the hexagonal honeycomb is optimal among convex periodic tessellations of the plane, provided the cost functional is lower semicontinuous with respect to the Hausdorff convergence, and decreasing under Steiner symmetrization. \end{abstract} 
 
\title{A note on the honeycomb optimality \\ among periodic convex  tilings}
\maketitle

\section{Introduction} 
Finding optimal planar tilings, and in particular investigating the optimality of the honeycomb structure made by a packing of regular hexagons,  is a challenging and widely studied problem, both  for  geometric  and variational functionals.
For the classical De Giorgi perimeter,  in the seminal paper  \cite{hales}   Hales proved   the optimality of the honeycomb in full generality, namely among all possible coverings of the plane made by cells of equal area  (see also  \cite{morgansolo}).  For the Cheeger constant, the result is proved in  \cites{BBFV, BF1} 
(respectively with and without convexity constraint), while the cases of the first Robin Laplacian eigenvalue  and of the Robin torsional rigidity, have been settled  in \cite{BF2} (under convexity constraint).   For the first Dirichlet Laplacian eigenvalue, 
the optimality of the honeycomb  corresponds to a conjecture formulated by Caffarelli and Lin  \cite{CaffLin}, and remains, to the best of our knowledge, open.

In the recent paper \cite{cnf1}, we proved the existence of an optimal periodic tessellation minimizing quite general nonlocal perimeter functionals defined in terms of  interaction kernels, and we discussed the possible optimality of the honeycomb in the planar case, based on a symmetrization procedure. 
Since the description  of such procedure, as it was sketched in \cite[Lemma 5.1]{cnf1}, contained  some inaccuracies (though all the statements in \cite{cnf1} are correct), a first goal of this note is to fix rigorously its proof; a second goal is to observe that this procedure provides the optimality of the honeycomb for a much wider class of optimal partition problems. 
%

More precisely, in this note we prove that every centrally symmetric hexagon and every parallelogram  can be transformed into the regular hexagon with equal area by applying iteratively appropriate Steiner symmetrizations  (see Section \ref{sec:steiner}). Since  the only convex polygons which tessellate the plane by translations are centrally symmetric hexagons and parallelograms, 
as shown in \cite{fedorov,mm}, this implies that, for any  functional which is lower semicontinuous with respect to   the Hausdorff convergence of convex sets and is  decreasing under Steiner symmetrization, the honeycomb is the optimal periodic convex tiling of the plane. Plenty of geometric or variational energies  fit into this general framework (see Section \ref{sec:main}); moreover,  for some of them, the honeycomb  turns out to be the  {\it unique} optimal tiling. 

Several challenging related questions remain open. In particular is it natural to ask
whether, for all the energy functionals considered in this note, which include the first Dirichlet Laplacian eigenvalue, the honeycomb is optimal also without the periodicity and the convexity assumptions, corresponding to  
a generalized version of the  conjecture by Caffarelli and Lin \cite{CaffLin}.    A related  interesting open question concerns on the other hand  the possible existence of other local minimizers for such energies.    

A first step in the direction of proving the conjecture   would be establishing the existence of an optimal tessellation, by analogy to what was done in \cite{cnf1}  for nonlocal perimeters.    We recall that in \cite{cnparti}, a weaker  version of this conjecture was considered in the case of  anisotropic local perimeters under periodicity constraints: more precisely  it is proved that the optimal periodic planar tilings are given by centrally symmetric convex hexagons (eventually degenerating to parallelograms)  in the case of strictly convex anisotropy, and   by nondegenerate   convex hexagons  when the anisotropy is also differentiable.

\section{Steiner symmetrization for periodic convex tilings}\label{sec:steiner}

Let $\mathcal{H}$ denote the family of convex polygons which tessellate the plane  by translations. Recall that  such  a convex polygon  can only be either a centrally symmetric hexagon or a parallelogram \cite{mm, fedorov}. 
 The following lemma provides  a   constructive symmetrization procedure on the class $\mathcal{H}$.

For convenience of the reader, let us recall the classical definition of Steiner symmetrization of a set  $E$ 
belonging  more generally to the class  $\mathcal{C}$ of compact convex sets in $\R^2$ with nonempty interior, 
with respect to a straight line $\pi\subseteq\R^2$ passing through the origin. Denoting by  
$v$ the normal to $\pi$, and setting $L_{x}=\{x+tv \, , \,  t\in\R\}$ for $x\in \pi$,   the Steiner symmetrization of $E$ (see \cite{barn}) is given by 
\begin{equation}\label{f:steiner} 
S_\pi(E)=\bigcup_{x\in \pi, L_x\cap E\neq \emptyset}  \{x+rv, \ 2|r|\leq \mathcal{H}^1(E\cap L_x)\}. 
\end{equation} 
It is straightforward to check that, if $E\in \mathcal{C}$, then $S_\pi(E)\in\mathcal{C}$, with $|S_\pi(E)|=|E|$.

 \begin{lemma}\label{lemmahex} 
Every $H_0  \in \mathcal H$ can be transformed  into  a regular hexagon,  by using 
at most countable  Steiner symmetrizations.  \EEE 
\end{lemma} 
 
\begin{proof}
Starting from a fixed hexagon $H_0 \in \mathcal H$ (possibly degenerated into a parallelogram), let us construct a sequence of centrally symmetric  hexagons $H_n$ (not degenerating into parallelograms), such that $H_n$ converge in Hausdorff distance to a regular hexagon $H^*$. We proceed by steps.  

\smallskip 
{\bf Step 1: inizialization of the symmetrization procedure.}

Let denote by $A_0,B_0,C_0,D_0,E_0,F_0$ the vertices of $H_0$, ordered  in counterclockwise sense. We apply to  $H_0$ a Steiner symmetrization with respect of the axis of a diagonal
joining the  vertices of two consecutive sides, say $A_0 E_0$.
By this way, $H_0$ is transformed into a 
new hexagon   $H_1$, whose vertices $A_1,B_1,C_1,D_1,E_1,F_1$ are not aligned three by three,
namely $H _1$ is not degenerated into a parallelogram (see Figure \ref{figura1}, where $H_0$ is represented in dotted line, and $H _1$ in continuous line; notice that $H _1$ is not degenerated into a parallelogram even if this is the case for $H _0$, see \cite[Figure 1]{cnf1}). 
 Precisely, we have:
$A_1 = A_0$, $E_1 = E_0$, the triangle $A_0 E_0 F_0 $ becomes the isosceles triangle $A_1 E_1 F_1$, the sides $A_0B_0$ and $D_0E_0$ are transformed into the sides $A _1 B_1$ and $D_1E_1  $, which are parallel to the symmetry axis, and finally the triangle $B_0 C_0 D _0 $ becomes the isosceles triangle $B_1 C_1 D_1$. 
In particular,  the pairs  $(A_1B_1, D_1 E_1  )$, $(A_1F_1,  E_1 F_1)$ and $(B_1 C_1, C_1D_1)$ are congruent and, by symmetry with respect to the origin, also the sides $A_1F_1$ and $C_1D_1$ are congruent, as well as the sides
$B_1 C_1$ and $E_1F_1$. 

Thus, $H_{1}$ has $4$ congruent  sides of length $a_1= \overline{ A_1F_1} =\overline{E_1 F_1} = \overline{C_1 D_1}= \overline{B_1C_1}$, and $2$ congruent sides of length $b_1=\overline{A_1B_1} =\overline{D_1 E_1}$.  Moreover, since Steiner symmetrization perserves the area, we have that
$|H|=|H_0|$.    
\vspace{0.2 cm} 
 \begin{figure} [h] 
\centering   
\def\svgwidth{7cm}   
\begingroup%
  \makeatletter%
  \providecommand\color[2][]{%
    \errmessage{(Inkscape) Color is used for the text in Inkscape, but the package 'color.sty' is not loaded}%
    \renewcommand\color[2][]{}%
  }%
  \providecommand\transparent[1]{%
    \errmessage{(Inkscape) Transparency is used (non-zero) for the text in Inkscape, but the package 'transparent.sty' is not loaded}%
    \renewcommand\transparent[1]{}%
  }%
  \providecommand\rotatebox[2]{#2}%
  \ifx\svgwidth\undefined%
    \setlength{\unitlength}{265.53bp}%
    \ifx\svgscale\undefined%
      \relax%
    \else%
      \setlength{\unitlength}{\unitlength * \real{\svgscale}}%
    \fi%
  \else%
    \setlength{\unitlength}{\svgwidth}%
  \fi%
  \global\let\svgwidth\undefined%
  \global\let\svgscale\undefined%
  \makeatother%
  \begin{picture}(5,0.5)%
    \put(0.6,-0.1){\includegraphics[height=4.5cm]{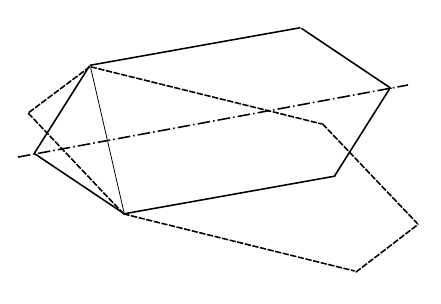} 
  }    
    \put(0.67, 0.05){\color[rgb]{0,0,0}\makebox(0,0)[lb]{\smash{$A_0=A_1$}}}
    \put(0.67, 0.44){\color[rgb]{0,0,0}\makebox(0,0)[lb]{\smash{$E_0=E_1$}}}
    \put(0.6, 0.3){\color[rgb]{0,0,0}\makebox(0,0)[lb]{\smash{$F_0$}}}
        \put(0.58, 0.2){\color[rgb]{0,0,0}\makebox(0,0)[lb]{\smash{$F_1$}}}
 \put(1.33, -0.08){\color[rgb]{0,0,0}\makebox(0,0)[lb]{\smash{$B_0$}}}
 \put(1.49, 0.04){\color[rgb]{0,0,0}\makebox(0,0)[lb]{\smash{$C_0$}}}
 \put(1.3, 0.12){\color[rgb]{0,0,0}\makebox(0,0)[lb]{\smash{$B_1$}}}
  \put(1.41, 0.385){\color[rgb]{0,0,0}\makebox(0,0)[lb]{\smash{$C_1$}}}
    \put(1.3, 0.27){\color[rgb]{0,0,0}\makebox(0,0)[lb]{\smash{$D_0$}}} 
    \put(1.24, 0.5){\color[rgb]{0,0,0}\makebox(0,0)[lb]{\smash{$D_1$}}} 
     \end{picture}%
\endgroup%
\vskip .5 cm
\caption{}
\label{figura1}   
\end{figure} 

{\bf Step 2: construction of the sequence of hexagons $H_n$.}

We proceed by applying the same procedure as in Step 1 to  $H_1$, by symmetrizing with respect to the axis of the diagonal
connecting  a side of length $a_1$ and a side of length $b_1$, say e.g. the diagonal  $D_1F_1$.  We 
thus obtain a new hexagon  $H_{2}$, having  $4$ congruent sides of length $a_2$ and $2$ congruent sides of length $b_2$ (see Figure \ref{figura2}, where $H_1$ is represented in dotted line, and $H _2$ in continuous line).  

We point out that  the reason for displaying this second step, is that we aim at comparing the lengths of the sides of $H _1$ with the lengths of the sides of $H _2$ (comparison that we are going to  extend  iteratively to the next hexagons in the construction). 
  
 \vspace{0.2 cm} 
 \begin{figure} [h] 
\centering   
\def\svgwidth{7cm}   
\begingroup%
  \makeatletter%
  \providecommand\color[2][]{%
    \errmessage{(Inkscape) Color is used for the text in Inkscape, but the package 'color.sty' is not loaded}%
    \renewcommand\color[2][]{}%
  }%
  \providecommand\transparent[1]{%
    \errmessage{(Inkscape) Transparency is used (non-zero) for the text in Inkscape, but the package 'transparent.sty' is not loaded}%
    \renewcommand\transparent[1]{}%
  }%
  \providecommand\rotatebox[2]{#2}%
  \ifx\svgwidth\undefined%
    \setlength{\unitlength}{265.53bp}%
    \ifx\svgscale\undefined%
      \relax%
    \else%
      \setlength{\unitlength}{\unitlength * \real{\svgscale}}%
    \fi%
  \else%
    \setlength{\unitlength}{\svgwidth}%
  \fi%
  \global\let\svgwidth\undefined%
  \global\let\svgscale\undefined%
  \makeatother%
 \begin{picture}(5,0.5)%
    \put(0.6,-0.1){ \includegraphics[height=4.5cm]{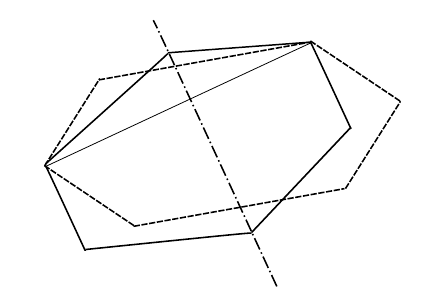}}    
  \put(0.82, 0.05){\color[rgb]{0,0,0}\makebox(0,0)[lb]{\smash{$A_1$}}}
    \put(0.76, 0.4){\color[rgb]{0,0,0}\makebox(0,0)[lb]{\smash{$E_1$}}}
        \put(0.5, 0.2){\color[rgb]{0,0,0}\makebox(0,0)[lb]{\smash{$F_1=F_2$}}}
        \put(1.35, 0.12){\color[rgb]{0,0,0}\makebox(0,0)[lb]{\smash{$B_1$}}}
  \put(1.46, 0.34){\color[rgb]{0,0,0}\makebox(0,0)[lb]{\smash{$C_1$}}}
    \put(1.24, 0.48){\color[rgb]{0,0,0}\makebox(0,0)[lb]{\smash{$D_1=D_2$}}} 
    \put(0.73, -0.01){\color[rgb]{0,0,0}\makebox(0,0)[lb]{\smash{$A_2$}}}
       \put(1.17, 0.03){\color[rgb]{0,0,0}\makebox(0,0)[lb]{\smash{$B_2$}}}
          \put(1.355, 0.28){\color[rgb]{0,0,0}\makebox(0,0)[lb]{\smash{$C_2$}}}
           \put(0.97, 0.465){\color[rgb]{0,0,0}\makebox(0,0)[lb]{\smash{$E_2$}}}
 \end{picture}%
\endgroup%
\vskip .5 cm
\caption{}
\label{figura2}   
\end{figure}

First of all we observe that, since Steiner symmetrization decreases perimeter, we have  that 
$$4a_1+2b_1\geq 4a_2+2 b_2\,.$$ 
  Next we observe that 
the triangle $D_1E_1F_1$ becomes the isosceles triangle $D_2E_2F_2$, which has the same basis 
$D_2  F_2 = D_1 F_1$, and
 two congruent sides of length $a_2$. Thus,  applying again the 
 monotonicity of perimeter under Steiner symmetrization, we obtain 
 $$a_1+b_1\geq 2 a_2\,.$$  
Finally we look at the vertices $A_1, B _ 1, C_1$: they are turned into the new vertices $A_2, B_2, C_2$,  which satisfy 
$ \overline { A_2B_2} = \overline {B _2 C_2} = a_2$, and $\overline {A_2 F_2 } = \overline{C_ 2  D_2 } =b_2$.  Since the diagonal $A_1C_1$ is parallel to the diagonal $D_1F_1$, and since we are symmetrizing with respect to the axis of $D_1F_1$, the side $A_2F_2$ is the orthogonal projection of the side $A_1F_1 $ onto the parallel to the symmetrization axis passing through $F_1$. In particular, due to this fact, we see that 
$$b_2\leq a_1\,.$$    

We now continue by applying recursively the same kind of symmetrization procedure, obtaining by this way a sequence of centrally symmetric hexagons $H_n$,   with $|H_n|=|H_0|$,  each one with  $2$ congruent sides of length $b_n$, $4$ congruent sides of length $a_n$, and such that the following conditions hold
for every $n\geq 1$: 

\smallskip
\begin{enumerate}\item[(i)] $4a_{n+1}+2 b_{n+1}\leq 4a_n+2b_n$,

\smallskip
\item[(ii)] $ 2 a_{n+1}\leq a_n+b_n $,

\smallskip
\item[(iii)] $b_{n+1}\leq a_n$. 
\end{enumerate} 

\smallskip 
{\bf Step 3: convergence of the sequence.}

For every $n \geq 1$, we define $c_n=\max(a_n, b_n)$ and $d_n=\min (a_n, b_n)$. Note that $0<d_n\leq c_n$ for every $n \geq 1$. We claim that 
\begin{equation}\label{f:claim}
\lim_n c_n =  \lim_n d_n   >0\,.
\end{equation} 
Once proved this claim, it is immediate to conclude that $H_n$ is converging in measure 
(or equivalently in Hausdorff distance) 
to the regular hexagon $H^*$, with sides of length equal to  any of the limits in \eqref{f:claim}. 
We are left to prove \eqref{f:claim}.  By properties (ii),(iii) above, and by the  definition of $ c_n$,  we have that
\begin{equation}\label{uno} 2 a_{n+1}\leq a_n+b_n \leq 2 c_n\qquad  b_{n+1}\leq a_n\leq c_n.\end{equation} 
From the above inequalities and the definition of $c _{n+1}$, we infer that  
\[c_{n+1}\leq c_n,\] 
namely the sequence $c_n$ is monotone non increasing, so that
\begin{equation}\label{limit} \lim_n c_n = \inf _n c_n =: c.\end{equation}  
Notice that necessarily $c>0$, since the area of each  $H_n$ is constantly equal to the area of $H_0$.   

Let now $\eps>0$ be fixed, and let $n_\eps$ be such that $c_n\leq c+\eps$ for $n>n_\eps$.  

We  are going to show that
\begin{equation}\label{f:claim2} 
\max(d_n, d_{n+1})\geq c-\eps \qquad \forall n > n _ \eps\,.
\end{equation}  
In turn, this implies that 
$$c-\eps\leq\max(d_n, d_{n+1})\leq c_n\leq c+\eps \qquad \forall n>n_\eps \,,$$
which ensures that $\lim d_n=c$ (and hence that  our claim 
\eqref{f:claim} holds true).

In order to prove \eqref{f:claim2}, we consider any $n>n_\eps$, and we distinguish two possible situations: either $c_n=b_n$ and $d_n=a_n$, or  $c_n=a_n$ and $d_n=b_n$.

\begin{itemize}
\item[--] 
Case $c_n=b_n$ and $d_n=a_n$.  By properties (ii)-(iii) in Step 2, we have 
$$2a_{n+1}\leq a_n+b_n= c_n+d_n \qquad \text{ and } \qquad b_{n+1}\leq a_n=d_n\leq \frac{c_n+d_n}{2} \,.$$
Hence
\[c\leq c_{n+1}\leq \frac{c_n+d_n}{2}\leq \frac{c+\eps+d_n}{2}\,,\]
yielding that $d_n\geq c-\eps$. 

\smallskip
\item[--] 
Case $c_n=a_n$ and $d_n=b_n$.  We distinguish two further sub-cases. 

If $c_{n+1}=a_{n+1}$, by property (ii), we get 
$$2c\leq 2c_{n+1}=2a_{n+1}\leq c_n+d_n\leq c+\eps+d_n \,,$$
yielding that  $d_n\geq c-\eps$. 

If $c_{n+1}=b_{n+1}$, by properties (ii)-(iii) in Step 2,  we get
$$\begin{cases}
2a_{n+2}\leq a_{n+1}+b_{n+1}= c_{n+1}+d_{n+1} & \\ 
\noalign{\medskip}
b_{n+2}\leq a_{n+1}=d_{n+1}\leq \frac{c_{n+1}+d_{n+1}}{2} \,,& 
\end{cases} $$ 
so that 
\[c\leq c_{n+2}\leq \frac{c_{n+1}+d_{n+1}}{2}\leq \frac{c+\eps+d_{n+1}}{2}\,, \]
 yielding that $d_{n+1}\geq c-\eps$. 
 \end{itemize}
 
 \end{proof}

 \section{The optimality of the honeycomb}\label{sec:main}

We now consider the isoperimetric problem 
 \begin{equation}\label{iso2poly} 
\inf \Big \{  \mathcal{F}(E) \ :\ E \in \mathcal H\,, \ |E|=1 \Big \}\,, 
 \end{equation}  
 where $\mathcal{H}$ denotes as in the previous section the class of   convex polygons which tessellate the plane  by translations, and $\mathcal F $ belongs to a broad class of cost functionals.  Precisely, we assume that 
 \[\mathcal{F}:\mathcal{H}\to \R\cup\{+\infty\},\]
 satisfies the following conditions:
\begin{itemize}
\item[--] if $H_n$ is a sequence in $\mathcal{H}$ converging to  $H $ in Hausdorff distance, then \begin{equation}\label{lsc}\liminf_n \mathcal{F}(H_n)\geq \mathcal{F}(H);
\end{equation}
\item[--] for every line $\pi$ and every $H\in \mathcal{\mathcal H}$, if
$S _\pi ( H)$ denotes the Steiner symmetrization of $H$ with respect to the line $\pi$ (see \eqref{f:steiner}),  
 it holds \begin{equation}\label{steiner}\mathcal{F}(S_\pi(H))\leq \mathcal{F}(H).\end{equation} 
\end{itemize}

 \bigskip 
As a direct consequence of Lemma \ref{lemmahex},  we obtain:

 \begin{theorem}\label{thm1} 
For any functional $\mathcal{F}$   satisfying  \eqref{lsc} and \eqref{steiner},    there holds 
$$\min_{H\in \mathcal{H}} \mathcal{F}(H) =  \mathcal{F}(H^*)\,,$$
 where $H^*$ is the regular hexagon with unit area. 
Moreover, if \eqref{steiner} holds as a strict inequality when $H \neq S_\pi(H)$,  then  $H^*$ is the unique minimizer.
\end{theorem} 
\begin{proof}  For every $H\in\mathcal{H}$, we let $H _0 = H$  and we construct the sequence $H_n$ as in Lemma \ref{lemmahex}.  Then by \eqref{steiner} we get that $\mathcal{F}(H_{n+1})\leq \mathcal{F}(H_{n})\leq \mathcal{F}(H)$ for every $n\geq 1$. By \eqref{lsc} we have
$\mathcal{F}(H^*)\leq \liminf _n \mathcal{F}(H_{n})\leq \mathcal{F}(H)$, where $H^*$ is the regular hexagon with area $1$. 
Moreover, we have the strict inequality $\mathcal{F}(H^*) <  \mathcal{F}(H)$ in case  \eqref{steiner} is strict when
$H\neq S_\pi(H)$. 
\end{proof} 

\bigskip
 
To conclude,  we point out that
 the class of  
functionals satisfying assumptions \eqref{lsc}-\eqref{steiner} includes plenty of 
 geometric or variational energies, other than
 the classical De Giorgi perimeter  and the Cheeger constant
  (for which as mentioned in the Introduction the optimality of the honeycomb is already known in a much broader sense, without the periodicity and convexity constraints). 
  A non hexaustive list of such functionals
is given below.

\begin{itemize}
\item Nonlocal perimeter functionals  of the kind
\[\Per_K(E)=\int_{E}\int_{\R^2\setminus E} K(x-y)dxdy,\] for any kernel $K$
which is radial,  nonnegative,  with $ \min \{ 1, |h| \} K (| h|) \in L^ 1 (\R ^2)$,  and $\liminf_{z\to 0^+}\big [  K(|z|)-K(|z+x|) \big ] >0$ for  $x \neq 0$, see \cite{cnf1}. Note that if $K$ is strictly decreasing, then \eqref{steiner} holds as a strict inequality for $H\neq S_\pi(H)$. 
\item Riesz-type energies of the form \[\mathcal{R}(E)=-\int_E\int_E K(x-y)dxdy,\] for any kernel $K$ which is   radial, decreasing, and positive definite,  with $\int_0^1 K(r)dr<+\infty$ (see \cite{lieb} and the related papers 
 \cite{BCT,BBF}). 
Also in this case, if $K$ is strictly decreasing, \eqref{steiner} holds as a strict inequality for $H\neq S_\pi(H)$. 

\smallskip
\item  The logaritmic capacity:
\[\text{LogCap}(E)=e^{-W(E)}\,, \qquad\text{where } W(E)=\inf_{\mu\in \mathcal{P}_E} \int_{E}\int_{E} \log(|x-y|)d\mu(x)d\mu(y),\] 
being $\mathcal{P}_E$ the class of Borel probability measures supported in $E$,  see \cite[Theorem 6.29]{barn}.

\smallskip
\item The Riesz-$\alpha$-capacity for any $\alpha\in (0,2)$:  \[\mathcal{I}_\alpha(E)=\frac{1}{W_\alpha(E)}\,, \qquad\text{where }W_\alpha(E)= \inf_{\mu\in \mathcal{P}_E} \int_{E}\int_{E} |x-y|^{\alpha-2} d\mu(x)d\mu(y)\] where $\mathcal{P}_E$ is as above, see \cite{lieb}, \cite[Theorem 6.29]{barn}.

\smallskip
\item The variational $p$-capacity for any $p\in (1,2)$:  \[\text{Cap}_p(E)=\inf\{ \int _ {\R ^ 2} |\nabla u| ^ p  \ :\ u\in C ^ \infty_0 (\R^2)\, , \ u=1 \text{ on }E\},\] see \cite[Chapter 6]{barn}. 

\smallskip
\item The first Dirichlet Laplacian eigenvalue: \[\lambda_1(E)= \inf\left\{\int_E |\nabla u|^2 dx \ :\ u\in H _0 ^{1 }(E), \ \int_E | u|^2 dx =1\right\}\]  see \cite[Chapter 3]{hen}.
 
\end{itemize}


\begin{bibdiv}
\begin{biblist}
 \bib{barn}{book}{
AUTHOR = {Baernstein, A.},
     TITLE = {Symmetrization in analysis},
    SERIES = {New Mathematical Monographs},
    VOLUME = {36},
      NOTE = {With David Drasin and Richard S. Laugesen,
              With a foreword by Walter Hayman},
 PUBLISHER = {Cambridge University Press, Cambridge},
      YEAR = {2019},
     PAGES = {xviii+473},
}

\bib{BBF}{article}{
    AUTHOR = {Bogosel, B.},
    AUTHOR = {Bucur, D.},
    AUTHOR = {Fragal\`a, I.},
     TITLE = {The nonlocal isoperimetric problem for polygons:  
Hardy-Littlewood and Riesz inequalities 
},
   JOURNAL = {Math. Ann.}
      YEAR = {2023},

   }
   
   
    
\bib{BCT}{article}{
    AUTHOR = {Bonacini, M.},
AUTHOR = {Cristoferi, R.},
AUTHOR = {Topaloglu, I.},
     TITLE = {Riesz-type inequalities and overdetermined problems for
              triangles and quadrilaterals},
   JOURNAL = {J. Geom. Anal.},
    VOLUME = {32},
      YEAR = {2022},
    NUMBER = {2},
     PAGES = {Paper No. 48},
     }
		
\bib{BBFV}{article}{
    AUTHOR = {Bucur, D.},
    AUTHOR = {Fragal\`a, I.}, 
    AUTHOR = {Velichkov, B.},
    AUTHOR = {Verzini, G.},
     TITLE = {On the honeycomb conjecture for a class of minimal convex partitions 
},
   JOURNAL = {Trans. Amer. Math. Soc.}
   VOLUME = {370},
      YEAR = {2018},
    NUMBER = {10},
     PAGES = {7149--7179},
   
   }
   
   
   
\bib{BF1}{article}{
    AUTHOR = {Bucur, D.},
    AUTHOR = {Fragal\`a, I.}, 
     TITLE = {Proof of the honeycomb asymptotics for optimal Cheeger clusters
},
   JOURNAL = {Adv. Math.}
   VOLUME = {350},
      YEAR = {2019},
     PAGES = {97--129},
   
   }
   
   
      
\bib{BF2}{article}{
    AUTHOR = {Bucur, D.},
    AUTHOR = {Fragal\`a, I.}, 
     TITLE = {On the honeycomb conjecture for Robin Laplacian eigenvalues},
   JOURNAL = {Commun. Contemp. Math.}
   VOLUME = {21},
      YEAR = {2019},
    NUMBER = {2},
     PAGES = {29 pp.},
 
   }

   
   	\bib{CaffLin}{article}{
	Author = {Caffarelli, L. A.},
	Author = {Lin, F. H.},
	Fjournal = {Journal of Scientific Computing},
	Journal = {J. Sci. Comput.},
	Number = {1-2},
	Pages = {5--18},
	Title = {An optimal partition problem for eigenvalues},
	Volume = {31},
	Year = {2007}}

   
		

      
 \bib{cnf1}{article}{
 author={Cesaroni, A.},
 author={ Fragalà, I.},
 author={ Novaga, M},
 title={ Lattice tilings minimizing nonlocal perimeters},
 journal={Commun. Contemp. Math. },
 volume={27},
 year={2025},
 number={6}, 
 pages={Paper No. 2450043},
 }
     
\bib{cnparti}{article}{
    AUTHOR = {Cesaroni, A.},
    author={Novaga, M.},
  TITLE = {Periodic partitions with minimal perimeter},
   JOURNAL = {Nonlinear Anal.},
    VOLUME = {243},
      YEAR = {2024},
     PAGES = {Paper No. 113522, 16},   }
 	



 
 
 
\bib{fedorov}{article}{ 
AUTHOR = {Fedorov, E.S.},
TITLE = {The Symmetry of Regular Systems of Figures}, 
JOURNAL = {Proceedings of the Imperial St. Petersburg Mineralogical Society}, 
VOLUME = {28}, 
YEAR = {1891},
PAGES = {1--146},
} 
\bib{Gal14}{article}{ 
AUTHOR = {Gallagher, P.},
 author={Ghang, Whan},  
 author={Hu, David},
 author={Martin, Zane}, 
 author={Miller, Maggie},
 author={Perpetua, Byron}, 
 author={Waruhiu, Steven},
 title={Surface-area-minimizing n-hedral Tiles},
journal={Rose-Hulman Undergraduate Mathematics Journal}, 
volume={15},
year={2014},
NUMBER = {1},
PAGES = {Article 13},
}



\bib{glasner}{article}{
    AUTHOR = {Glasner, K.},
     TITLE = {Segregation and domain formation in non-local multi-species
              aggregation equations},
   JOURNAL = {Phys. D},
    VOLUME = {456},
      YEAR = {2023},
     PAGES = {Paper No. 133936, 11},
     }



\bib{hales}{article}{
author={ Hales, T.C.},
title={The honeycomb conjecture},
journal={Discrete Comput. Geom.},
volume={ 25},
 YEAR = {2001},
    NUMBER = {1},
     PAGES = {1--22},
}
\bib{hen}{book}{
AUTHOR = {Henrot, A.},
     TITLE = {Extremum problems for eigenvalues of elliptic operators},
    SERIES = {Frontiers in Mathematics},
 PUBLISHER = {Birkh\"{a}user Verlag, Basel},
      YEAR = {2006},
     PAGES = {x+202},
}

\bib{lieb}{book}{
AUTHOR = {Lieb, E. H.},
author={Loss, M.},
     TITLE = {Analysis},
    SERIES = {Graduate Studies in Mathematics},
    VOLUME = {14},
   EDITION = {Second},
 PUBLISHER = {American Mathematical Society, Providence, RI},
      YEAR = {2001},
     PAGES = {xxii+346},
     }

 
\bib{maggi}{book}{
    AUTHOR = {Maggi, F.},
     TITLE = {Sets of finite perimeter and geometric variational problems},
    SERIES = {Cambridge Studies in Advanced Mathematics},
    VOLUME = {135},
 PUBLISHER = {Cambridge University Press, Cambridge},
      YEAR = {2012},
}

 	
      
     
\bib{mm}{article}{
    AUTHOR = {McMullen, P.},
     TITLE = {Convex bodies which tile space by translation},
   JOURNAL = {Mathematika},
    VOLUME = {27},
      YEAR = {1980},
    NUMBER = {1},
     PAGES = {113--121},
}
	

	
	

\bib{morgansolo}{article}{
    AUTHOR = {Morgan, F.},
     TITLE = {The hexagonal honeycomb conjecture},
   JOURNAL = {Trans. Amer. Math. Soc. },
    VOLUME = {351},
      YEAR = {1999},
    NUMBER = {5},
     PAGES = {1753--1763},
}


 
%

\bib{thompson}{article}{ 
AUTHOR = {Thomson (Lord Kelvin), W.},
     TITLE = {On the division of space with minimum partitional area},
   JOURNAL = {Acta Math.},
    VOLUME = {11},
      YEAR = {1887},
    NUMBER = {1-4},
     PAGES = {121--134},
}
%
\bib{kelvin}{book}{
EDITOR = {Weaire, D.},
     TITLE = {The {K}elvin problem},
         NOTE = {Foam structures of minimal surface area},
 PUBLISHER = {Taylor \& Francis, London},
      YEAR = {1996},
      }
 
%
\end{biblist}\end{bibdiv}
\end{document}